\documentclass[preprint]{elsarticle}
\usepackage{amsbsy}
\usepackage{amsfonts}
\usepackage{amsmath}
\usepackage{amssymb}

\newdefinition{definition}{Definition}
\newtheorem{theorem}{Theorem}
\newtheorem{proposition}[theorem]{Proposition}
\newproof{proof}{Proof}

\newtheorem{lemma}{Lemma}

\newdefinition{remark}{Remark}

\newdefinition{example}{Example}

\newcommand{\no}{{\mathbb{N}}}

\begin{document}

\begin{frontmatter}
\title{The classification of abelian groups generated by time-varying automata and by Mealy automata  over the binary alphabet}
\author[rvt]{Adam Woryna}
\ead{adam.woryna@polsl.pl}
\address[rvt]{Silesian University of Technology, Institute of Mathematics, ul. Kaszubska 23, 44-100 Gliwice, Poland}

\begin{abstract}
For every natural number $n$, we classify abelian groups generated by an $n$-state time-varying automaton over the binary alphabet, as well as by an $n$-state Mealy automaton over the binary alphabet.
\end{abstract}

\end{frontmatter}

\section{Introduction}

In the theory of computation  time-varying automata over a finite alphabet are finite-state transducers which constitute a natural generalization of  Mealy-type automata as they allow to change both the transition function and the output function in successive steps of processing input sequences of letters  into output sequences (see~\cite{1}). These automata in turn constitute a subclass in the class of time-varying automata over a changing alphabet (see~\cite{4}). In group theory all these types of transducers turned out to be  a useful tool for defining and studying groups of automorphisms of certain rooted trees, which helped in the discovery of interesting  geometric and algebraic properties and the dynamics of various types of groups associated with their actions on these trees. The intensive  study of Mealy automata in relation to  automorphism groups has  continued for last three  decades, and the extraordinary properties of these groups are discussed in  a great deal of papers (for the comprehensive works see~\cite{34,2,26,27,3}). On the other hand, the  investigation of automorphism groups as groups defined by time-varying automata is a quite new approach, which we introduced in~\cite{4} (for other results on this subject see~\cite{22,23,5,6,7,21}).

If an automaton $A$ is invertible, then each of its states defines an automorphism of the corresponding rooted tree. The group generated by  these automorphisms (i.e. by automorphisms corresponding to all the states of $A$) is called the group generated by the automaton $A$ and is usually denoted by $G(A)$.

\begin{definition}
Let $n$ and $k$ be natural numbers and let $G$ be an abstract group. We say that  $G$ is generated by an $n$-state time-varying (Mealy) automaton over a $k$-letter alphabet if there is an invertible time-varying (Mealy) automaton $A$ with an $n$-element set of states which works over a $k$-letter alphabet such that $G$ is isomorphic to the group $G(A)$ generated by this automaton.
\end{definition}

Given two natural numbers $n$, $k$ and a class of abstract groups,  the natural question arises: which groups from this class are generated by an $n$-state (time-varying or Mealy) automaton over a $k$-letter alphabet? This problem was  studied so far only in the case of Mealy-type automata and it turned out to be  difficult even for small values of $n$ and $k$ and for classes containing algebraically well known constructions, such as abelian groups, free groups, free products of finite groups, and many others.

It is worth noting that there is no general methods for deducing even the basic algebraic properties of the group generated by an automaton directly from the  structure of this automaton. For example, for a long time it was an open problem  whether a free non-abelian group is generated by a Mealy automaton  (see~\cite{25} solving this problem). However, a candidate for the solution was known and studied since 80's last century (the so-called Aleshin-Vorobets automaton --  see~\cite{8,12}). Until now, we do not know if the free non-abelian group of rank two is generated by a $2$-state Mealy automaton or by  a 2-state time-varying automaton over a finite alphabet. On the other hand, we provided  for this group a natural construction of a 2-state time-varying automaton over an unbounded changing alphabet (see~\cite{6}).

The full classification of groups generated by Mealy automata is known only in the simplest non-trivial case, i.e. when $n=k=2$ (\cite{2}). Apart from that, for  the classes of abelian groups and finite groups the solutions for  Mealy automata are known only in the case $(n,k)=(3,2)$ (see Theorems~3,4 in the extensive work~\cite{15} devoted to the classification of all groups generated by 3-state Mealy automata over the binary alphabet).

\section{The results}
The goal of this paper is  providing for every natural number $n$ the full classification of abelian groups generated by an $n$-state time-varying automaton over the binary alphabet as well as by an $n$-state Mealy automaton over the binary alphabet.

Let us denote the following classes of abelian groups:
\begin{itemize}
\item $\mathcal{TVA}(n)$ -- the class of abelian groups generated by an $n$-state time-varying automaton over the binary alphabet,
\item $\mathcal{MA}(n)$ -- the class of abelian groups generated by an $n$-state Mealy automaton over the binary alphabet,
\item $\mathcal{AB}_2(n)$ -- the class of  abelian groups of rank not greater than $n$ in which the torsion part is a 2-group,
\item $\mathcal{FA}(n)$ -- the class of  free abelian groups of rank not greater than $n$,
\item $\mathcal{EA}_2(n)$ -- the class of  elementary abelian 2-groups of rank not greater than $n$.
\end{itemize}

The main result of the paper  is the  following characterization.

\begin{theorem}\label{t1}
$\mathcal{TVA}(1)=\mathcal{EA}_2(1)$, and if $n>1$, then $\mathcal{TVA}(n)=\mathcal{AB}_2(n)$.
\end{theorem}

\begin{theorem}\label{t2}
$\mathcal{MA}(n)=\mathcal{FA}(n-1)\cup \mathcal{EA}_2(n)$. In particular,  there are exactly $2n$ abelian groups generated by an $n$-state Mealy automaton over the binary alphabet.
\end{theorem}

For the proof of Theorem~\ref{t1} (Section~\ref{sec5}), we find at first some restrictions on abelian groups generated by time-varying automata over the binary alphabet. These restrictions imply the inclusion $\mathcal{TVA}(n)\subseteq\mathcal{AB}_2(n)$ for each $n\geq 1$, as well as the equality $\mathcal{TVA}(1)=\mathcal{EA}_2(1)$. Next, for every $n\geq 1$ and every group $G\in \mathcal{AB}_2(n)$ of rank $n$, we provide an explicit  construction of a time-varying automaton $A$ over the binary alphabet  which generates $G$ (see Propositions~\ref{p1}--\ref{p3}).
In the  construction of the automaton $A$, we distinguish between three cases: where the group $G$ is cyclic, where  $G$ is non-cyclic and  has a non-trivial torsion part, and finally, where $G$ is a non-cyclic, torsion-free group. Note that in the last case $G$ must be a free abelian group, which follows from  the property that every finitely generated, abelian group is a direct sum of cyclic groups. To derive the required property of our construction, we use the language of wreath recursions, which is popular in the study of groups generated by automata, and  which arises from an embedding of the group $G(A)$  into the permutational wreath product (for more on the method involving wreath recursions in the case of Mealy automata see, for example, \cite{2,3}, and in the case of time-varying automata  -- see \cite{21}).

In the proof of Theorem~\ref{t2} (Section~\ref{sec6}), for  an arbitrary group $G$  generated by a Mealy automaton  over the binary alphabet, we deduce certain relations  between the first level stabilizer of $G$, the set $I_G$ of involutions  and the set $G^2$ of squares. We show (Proposition~\ref{pp}) a quite  unexpected dichotomy, which holds in the case when $G$ is  abelian. Namely,  the relations  imply in this case that  $G$  is either a free abelian group or an elementary abelian 2-group. Next,  we  introduce and study  a construction of the Mealy automata generating  elementary abelian 2-groups. We also  use the known construction of the Mealy  automata generating  free abelian groups (so-called ``sausage'' automata -- see \cite{2}), as well as the restriction for the rank of  free abelian groups generated by   Mealy automata over the binary alphabet (see Proposition~3.1 from~\cite{17}).

We hope that the above characterizations, together with the involved constructions and the methods presented in the proofs, will encourage  to develop further study of the groups generated by time-varying automata. This may be useful when investigating some important, computational problems, which are open in the class of groups generated by time-varying automata, but have a simple solution in the class of groups generated by Mealy automata.
Let us mention the  word problem for groups (which is decidable in the latter class). It is known that there is a time-varying automaton $A$ over an arbitrary, unbounded changing alphabet such that the group $G(A)$ has undecidable word problem. This follows from the fact that there exist finitely generated, residually finite groups with undecidable word problem, as well as from some unconstructive method of defining an automaton realization for an arbitrary, finitely generated, residually finite group (the so-called diagonal realization -- see \cite{4,7}). However, there is not known any  explicit and naturally defined  construction of  an automaton $A$ such that the group $G(A)$ has undecidable word problem.  Moreover, in the class of time-varying automata over a finite alphabet, even the question on the  existence of such an automaton is open.

In view of the above study, also  the following problem is interesting and far from trivial: given $n$ and $k$, how many (pairwise non-isomorphic) groups are contained in the class $\mathcal{GTV}(n,k)$ of the groups generated by an $n$-state time-varying automaton over a $k$-letter alphabet? Obviously, if  $n=1$ or $k=1$, then the class $\mathcal{GTV}(n,k)$ is finite. Further, if $n\leq n'$ and $k\leq k'$, then $\mathcal{GTV}(n,k)$ is contained in $\mathcal{GTV}(n',k')$ (see the last paragraph of Section~\ref{sec5} for the corresponding reasoning). Directly by Theorem~\ref{t1}, we obtain $\mathcal{AB}_2(2)\subseteq\mathcal{GTV}(2,2)$, which implies that $\mathcal{GTV}(2,2)$ contains infinitely many groups. But, if $k>1$ or $n>1$, then there are uncountable many $n$-state time-varying automata over a $k$-th letter alphabet. Thus  the natural problem arises: find the smallest $n$ and $k$ such that the class $\mathcal{GTV}(n,k)$ is uncountable. In particular, is the class $\mathcal{GTV}(2,2)$ uncountable?  Note that the class of all groups generated by an $n$-state Mealy  automaton over a $k$-letter alphabet is finite (there are exactly $(n^k\cdot k!)^n$ invertible Mealy automata with the $n$-element set of states over a $k$-th letter alphabet), which implies that the class of all groups generated by Mealy automata is countable.

The paper is organized as follows. Section~\ref{sec3} contains definitions of an automaton over a finite alphabet and the group generated by the automaton transformations of the corresponding rooted tree. In Section~\ref{sec4}, we describe the notion of the wreath recursion. In the last two sections, we present the proofs of our characterization.

\section{Automata and groups generated by automata}\label{sec3}

Let $\no=\{1,2,\ldots\}$ be the set of natural numbers, and let $X$ be a nonempty, finite set (a finite alphabet). A time-varying automaton over $X$ is a quadruple
$$
A=(X, Q, \varphi, \psi),
$$
where $Q$ is a finite set (set of states), $\varphi=(\varphi_i)_{i\in\no}$ is a sequence of so-called transition functions $\varphi_i\colon Q\times X\to Q$, and $\psi=(\psi_i)_{i\in\no}$ is a sequence of so-called output functions $\psi_i\colon Q\times X\to X$.  If the sequences $\psi$ and $\varphi$ are constant, then the automaton $A$ is called a Mealy automaton. In this case the sequences $\varphi$, $\psi$ are identified with, respectively, the transition function $\varphi_1$ and the output function $\psi_1$, i.e. we can write $A=(X, Q, \varphi_1, \psi_1)$. If for every $i\in\no$ and every $q\in Q$  the mapping $\sigma_{i,q}\colon x\mapsto \psi_i(q,x)$ (the so-called labeling of the state $q$ in the $i$-th transition of $A$) is a bijection on the alphabet, then the automaton $A$ is called invertible.

The tree $X^*$ of finite words over a finite alphabet $X$ consists of finite sequences of the form $x_1x_2\ldots x_l$ ($l\in\no$), where $x_i\in X$ ($1\leq i\leq l$), together with the empty sequence (empty word) denoted by $\epsilon$. The tree $X^*$ is an example of a rooted tree with the empty word as the root, and two words  are adjacent if and only if they are of the form $w$ and $wx$ for some $w\in X^*$, $x\in X$.

Let $A=(X, Q, \varphi, \psi)$ be a time-varying automaton over $X$. We consider the transformations
$$
q_i\colon X^*\to X^*,\;\;\; q\in Q,\;\; i\in\no
$$
defined in the following recursive way:
\begin{itemize}
\item $q_i(\epsilon)=\epsilon$ for all $q\in Q$, $i\in\no$,
\item $q_i(xw)=x'q'_{i+1}(w)$ for any $q\in Q$, $i\in\no$, $x\in X$, $w\in X^*$,
where $x'=\psi_i(q,x)=\sigma_{i,q}(x)$, $q'=\varphi_i(q,x)$.
\end{itemize}

The image $q_i(w)$ of any word $w\in X^*$ can be easily found directly from the diagram of the automaton $A=(X, Q, \varphi, \psi)$. We define  such a diagram  as an infinite, directed, locally finite graph  with the set $Q\times \mathbb{N}$ as the set of vertices. Each  vertex $(q,i)\in Q\times \mathbb{N}$  is labeled by the labeling $\sigma_{i,q}$ of the state $q$ in the $i$-th transition of $A$. Two vertices $(q,i), (q',i')\in Q\times \mathbb{N}$ are connected with a directed edge which starts at  $(q,i)$ and ends at $(q',i')$ if and only if $i'=i+1$ and there is  $x\in X$ such that $\varphi_i(q,x)=q'$; we label this edge by the  letter $x$. In particular, for every $i\in\no$, $q\in Q$ and $x\in X$, there is a unique edge which is labeled by $x$ and starts at $(q,i)$. Now, if $w=x_1\ldots x_l\in X^*$ is any word and if  $\tau_1, \ldots, \tau_l$ are the  labels of the consecutive vertices on the directed path which is labeled  by $w$ and which starts at  $(q,i)$, then we have $q_i(w)=\tau_1(x_1)\ldots\tau_l(x_l)$.

The transformation $q_i$ is called the automaton transformation corresponding to the state $q$ in its $i$-th transition. If the automaton $A$ is invertible, then each $q_i$ is an element of the automorphism group $Aut(X^*)$ of the tree $X^*$, i.e. $q_i$ is a permutation of the set $X^*$ preserving the root $\epsilon$ and the vertex adjacency. For every $i\in\no$ we denote by $G(A_i)$ the subgroup of $Aut(X^*)$ generated by automorphisms $q_i$ for $q\in Q$:
$$
G(A_i):=\langle q_i\colon q\in Q\rangle.
$$
The group
$$
G(A):=G(A_1)
$$
is called the group generated by the automaton $A$.

\section{The language of wreath recursions}\label{sec4}

Let $A=(Q, X, \varphi, \psi)$ be a time-varying automaton over a finite alphabet $X$. We define the transformations
$$
q_i|_w\colon X^*\to X^*,\;\;\;i\in\no,\;\;q\in Q,\;\;w\in X^*
$$
in the following recursive way:
\begin{itemize}
\item $q_i|_\epsilon=q_i$ for all $q\in Q$, $i\in\no$,
\item $q_i|_{xw}=q'_{i+1}|_w$ for all $q\in Q$, $i\in\no$, $x\in X$, $w\in X^*$,
where $q'=\varphi_i(q,x)$.
\end{itemize}
The transformation $q_i|_w$ is called the section of the transformation $q_i$ at the word $w$. Assuming that  $A$ is invertible, we can use the labelings of the states and the sections at one-letter words to describe the elements of the groups $G(A_i)$ ($i\in\no$) in the algebraical language of wreath products. Namely, if $X=\{x_1, \ldots, x_k\}$, then for every $i\in\no$ the mapping
$$
q_i\to (q_i|_{x_1}, \ldots, q_i|_{x_k})\sigma_{i, q}, \;\;\;q\in Q
$$
induces an embedding
$$
\phi_i\colon G(A_i)\to G(A_{i+1})\wr_X Sym(X)
$$
of the group $G(A_i)$ into the permutational wreath product of the group $G(A_{i+1})$ and the symmetric group of the alphabet, that is into the semidirect product
$$
G(A_{i+1})^{X}\rtimes Sym(X)
$$
with the natural action of $Sym(X)$  on the direct power  $G(A_{i+1})^X$ by permuting the factors.
Further, we identify any element $g\in G(A_i)$ with its image $\phi_i(g)$ and  call the relation $g=\phi_i(g)$ the wreath recursion of $g$. In particular, for the wreath recursion of any generator $q_i\in G(A_i)$ ($q\in Q$) we have:
$$
q_i=((r_1)_{i+1}, \ldots, (r_k)_{i+1})\sigma_{i,q},
$$
where $r_{j}=\varphi_i(q,x_j)$ for $1\leq j\leq k$. In general, if
$$
g=(h_1, \ldots, h_k)\pi
$$
is the wreath recursion of an arbitrary element $g\in G(A_i)$, then the element $h_j\in G(A_{i+1})$ ($1\leq j\leq k$), denoted further by $g|_{x_j}$, is called the section of $g$ at the letter $x_j\in X$, and the permutation $\pi$, denoted further by $\sigma_g$, is called the root permutation of $g$. In particular, for the wreath recursion of the inverse $g^{-1}$ we have
$$
g^{-1}=((g|_{y_{1}})^{-1}, \ldots,(g|_{y_{k}})^{-1})\pi^{-1},
$$
where $y_{j}={\pi^{-1}}(x_j)$ ($1\leq j\leq k$), and if an element $g' \in G(A_i)$ is defined by the wreath recursion
$$
g'=(g'|_{x_{1}}, g'|_{x_{2}},\ldots, g'|_{x_{k}}) \pi',
$$
then   the wreath recursion of the product $g\cdot g'$ satisfies:
$$
g\cdot g'=(g|_{x_{1}}\cdot g'|_{z_{1}}, \ldots, g|_{x_{k}}\cdot g'|_{z_{k}}) \pi\pi',
$$
where $z_{j}=\pi(x_j)$ for $1\leq j\leq k$.

\begin{remark}
If the root permutation $\sigma_g$ is trivial, then it is usually omitted in the wreath recursion, and we write $g=(g|_{x_1}, \ldots, g|_{x_k})$; also, if  the sections $g|_{x_j}$ ($1\leq j\leq k$) are all trivial, then $g$  is identified with its root permutation, and we write $g=\sigma_g$.
\end{remark}

We also define  the sections $g|_w$ ($g\in G(A_i)$, $w\in X^*$) recursively:
$g|_{\epsilon}=g$, $g|_{xw}=(g|_x)|_w$ for all $x\in X$, $w\in X^*$.

It is worth noting that if $A$ is a Mealy automaton, then for any $i\in\no$ and any $q\in Q$ the automaton transformation $q_i\colon X^*\to X^*$ of the state $q$ in its $i$-th transition coincides with the automaton transformation $q_1$. In particular, if $A$ is invertible, then $G(A_i)=G(A)$ for every $i\in\no$, and if $g\in G(A)$ and $w\in X^*$, then $g|_w\in G(A)$.

Given a nonempty word $w=x_1x_2\ldots x_l\in X^*$  and an element $g\in G(A)$, we can use the notions of a section and a root permutation  to compute the image $g(w)$ of $w$ under $g$ in the following way:
\begin{equation}\label{gw}
g(w)=g_0(x_1)g_1(x_2)\ldots  g_{l-1}(x_l)=\sigma_{g_0}(x_1)\sigma_{g_1}(x_2)\ldots \sigma_{g_{l-1}}(x_l),
\end{equation}
where $g_0:=g$ and $g_i:=g|_{x_1x_2\ldots x_i}$ for $1\leq i\leq l-1$.

\section{The proof of Theorem 1}\label{sec5}

The inclusion $\mathcal{TVA}(n)\subseteq \mathcal{AB}_2(n)$ follows from the obvious observation that the rank of a group generated by an $n$-state time-varying automaton is not greater than $n$, as well as from the following result

\begin{proposition}\label{inclu}
In the group $Aut(\{0,1\}^*)$ the order of any element of finite order is  a power of two.
\end{proposition}
\begin{proof}
Let $g\in Aut(\{0,1\}^*)$ be an element of a finite order $d:=o(g)$. For every $1\leq i< d$ there is a word $w_i\in \{0,1\}^*$ such that $g^i(w_i)\neq w_i$. Let $m=\max\{|w_i|\colon 1\leq i<d\}$ and let $\overline{g}$ be the restriction of $g$ to the subtree $T:=\{0,1\}^{\leq m}$ consisting of all binary words  of  length  not greater than $m$. Then $\overline{g}\in Aut(T)$ and $o(\overline{g})=o(g)=d$. It is well known that the group $Aut(T)$ is isomorphic to the $m$-iterated wreath product $C_2\wr C_2\wr\ldots\wr C_2$ of the cyclic group of order two, and that the order of this wreath product is equal to $2^{2^m-1}$. In particular, we have the divisibility $d| 2^{2^m-1}$, which implies the claim.\qed
\end{proof}

The next two propositions concern the generation of cyclic groups by time-varying automata over the binary alphabet equipped, respectively, with a single state and  two states.

\begin{proposition}
The only groups generated by a time-varying automaton over the binary alphabet with a single state is the trivial group and the group of order two, i.e. we have $\mathcal{TVA}(1)=\mathcal{EA}_2(1)$.
\end{proposition}
\begin{proof}
If $A=(\{0,1\}, \{a\}, \varphi, \psi)$ is a time-varying automaton with a single state $a$, then the corresponding automaton transformations $a_i\colon \{0,1\}^*\to \{0,1\}^*$ ($i\in\no$) satisfy the following wreath recursions: $a_i=(a_{i+1}, a_{i+1})\pi_i$, where $\pi_i\in Sym(\{0,1\})$. In particular, we have $a_i^2=(a_{i+1}^2, a_{i+1}^2)$ for every $i\in\no$, which implies: $a_i^2=id$ for every $i\in\no$. Now, if $\pi_i=id$ for every $i\in\no$, then the group $G(A)=\langle a_1\rangle $ is trivial, and if $\pi_i\neq id$ for some $i\in\no$, then $G(A)$ is the group of order two.\qed
\end{proof}

\begin{proposition}\label{p1}
Let $G$ be an infinite cyclic group $C_\infty$ or a finite cyclic group $C_{2^r}$ of order $2^r$, where $r\in\no$. Then there is a 2-state time-varying automaton $A$ over the binary alphabet such that $G\simeq G(A)$.
\end{proposition}
\begin{proof}
Let us define the subsets $N_1, N_2\subseteq \no$ as follows: $N_1=\{1,2,\ldots, r\}$, $N_2=\no$. For each $k\in\{1,2\}$ let
$$
A[k]=(\{0,1\}, \{a_1, a_2\}, \varphi, \psi)
$$
be the 2-state time-varying automaton over the binary alphabet in which the sequences $\varphi$, $\psi$ of transition and output functions are defined as follows:
\begin{eqnarray*}
\varphi_i(a_j,x)&=&\left\{
\begin{array}{ll}
a_1,& {\rm if}\;j=2,\;i\in N_k,\; x=1,\\
a_j, &{\rm otherwise},
\end{array}
\right.\\
\psi_i(a_j,x)&=&\left\{
\begin{array}{ll}
\tau(x),& {\rm if}\;j=2,\;i\in N_k,\\
x, &{\rm otherwise}
\end{array}
\right.
\end{eqnarray*}
for all $i\in\no$, $j\in\{1,2\}$, $x\in\{0,1\}$, where $\tau\in Sym(\{0,1\})$ is a transposition.

Let $a_{j,i}:=(a_{j})_i\in G(A[k]_i)$ ($j\in\{1,2\}$, $i\in\no$) be the automaton transformation of the state $a_j$ in its $i$-th transition.

In the case $k=1$ we obtain by the above formulae the following wreath recursions:
$$
a_{1,i}=(a_{1,i+1}, a_{1,i+1}),\;\;\;a_{2,i}=\left\{
\begin{array}{ll}
(a_{1, i+1}, a_{2,i+1})\tau, &{\rm if}\;1\leq i\leq r,\\
(a_{2,i+1}, a_{2,i+1}), &{\rm if}\; i>r.
\end{array}
\right.
$$
Hence $a_{1,i}=id$ for  $i\in\no$, $a_{2,i}=(id, a_{2,i+1})\tau$ for $1\leq i\leq r$, and $a_{2,i}=id$ for  $i>r$.
In particular, the order $o(a_{2,r})$ of the generator $a_{2,r}\in G(A[1]_r)$ is equal to 2, and for every $i\in\{1,\ldots, r\}$ we obtain by easy induction on $i$ the equality $o(a_{2, r+1-i})=2^i$. Hence $o(a_{2,1})=2^r$ and $G(A[1])=\langle a_{1,1}, a_{2,1}\rangle=\langle a_{2,1}\rangle\simeq C_{2^r}$.

In the case $k=2$ we have: $a_{1,i}=id$ for  $i\in\no$ and $a_{2,i}=(id, a_{2,i+1})\tau$ for $i\in\no$. From the last wreath recursion we obtain for any $i,s\in\no$ the following wreath recursion for the $s$-th power of $a_{2,i}$:
$$
a_{2,i}^s=(a_{2,i+1}^{\lfloor s/2\rfloor}, a_{2, i+1}^{\lceil s/2\rceil})\tau^s.
$$
In particular, if $a_{2,i}^s=id$ for some $i,s\in\no$, then $\tau^s=id$. Hence the number $s$ is even and $a_{2,i}^s=(a_{2,i+1}^{s/2}, a_{2, i+1}^{s/2})=id$. For every $j\in\no$ we obtain by easy induction on $j$ the divisibility $2^j\mid s$. In particular, the generator $a_{2,1}\in G(A[2])$ is of infinite order and $G(A[2])=\langle a_{1,1}, a_{2,1}\rangle=\langle a_{2,1}\rangle\simeq C_\infty$.\qed
\end{proof}

Further, let $n\geq 2$ and let $G$ be an arbitrary abelian group from the class $\mathcal{AB}_2(n)$ such that the rank of this group is equal to $n$. We provide the construction of an $n$-state time-varying automaton $A$ over the binary alphabet such that $G$ is isomorphic to the group $G(A)$ generated by this automaton. As we see in the following two propositions, our construction depends on whether $G$ contains a non-trivial torsion part.

\begin{proposition}\label{p2}
Let $n\in\no$ and let  $G$ be an abelian group of rank $n$ such that the torsion part of $G$ is a non-trivial 2-group. If $n\geq 2$, then there is an $n$-state time-varying automaton $A$ over the binary alphabet such that $G$ is isomorphic to  $G(A)$.
\end{proposition}
\begin{proof}
By the fundamental theorem of finitely generated abelian groups there are integers $1\leq d\leq n$ and $0\leq d'\leq n$ such that $d+d'=n$ and  the group $G$ is isomorphic to the direct sum
\begin{equation}\label{ds}
C_{2^{r_1}}\oplus \ldots\oplus C_{2^{r_d}}\oplus C_\infty^{d'}
\end{equation}
for some  $r_j\in\no$ ($1\leq j\leq d$), where $C_\infty^{d'}$ denotes the free abelian group of rank $d'$ in the case $d'>0$ or the trivial group in the case $d'=0$. Let us denote:
$$
M=\left\{\begin{array}{ll}
\no,&{\rm if}\;d'>0,\\
\{1,2,\ldots, R\},&{\rm if}\;d'=0,
\end{array}
\right.
$$
where $R=r_1+\ldots +r_d$, and let $N_1\cup\ldots\cup N_{n}$
be an arbitrary partition of the set $M$ in which:
$N_1=\{1,\ldots, r_1\}$, $|N_j|=r_j$ for $1<j\leq d$ and $|N_j|=\infty$ for $d<j\leq n$.

Let us assume that $n\geq 2$ and let $A$ be an $n$-state  time-varying automaton over the binary alphabet in which the set $Q$ of states consists of the symbols $a_1, \ldots, a_n$, and the sequences $\varphi=(\varphi_i)_{i\in\no}$, $\psi=(\psi_i)_{i\in\no}$ of transition and output functions are defined as follows:
\begin{eqnarray*}
\varphi_i(a_j,x)&=&\left\{
\begin{array}{ll}
a_1,&{\rm if}\;j\neq 1,\;i\in N_j,\;x=1,\\
a_2,&{\rm if}\;j=1,\;i\in N_1\setminus\{r_1\},\;x=1,\\
a_2,&{\rm if}\;j=1,\;i=r_1,\\
a_j,&{\rm otherwise},
\end{array}
\right.\\
\psi_i(a_j,x)&=&\left\{
\begin{array}{ll}
\tau(x),&{\rm if}\;i\in N_j,\\
x,&{\rm otherwise}
\end{array}
\right.
\end{eqnarray*}
for all $i\in\no$, $j\in\{1,2,\ldots, n\}$, $x\in\{0,1\}$, where $\tau\in Sym(\{0,1\})$ is a transposition. In the case $n=4$, $d=3$, $r_1=3$, $r_2=1$, $r_3=2$, $N_1=\{1,2,3\}$, $N_2=\{4\}$, $N_3=\{5,6\}$, $N_4=\mathbb{N}\setminus\{1,2,3,4,5,6\}$, the diagram of the automaton $A$ is depicted in Figure~\ref{fig} (for clarity, we replaced a large number of edges connecting  two given vertices and having the same direction with a one multi-edge and, since every multi-edge is labeled by the both letters of the alphabet, we omitted its labeling).
\begin{figure}
\centering
\includegraphics[width=13cm]{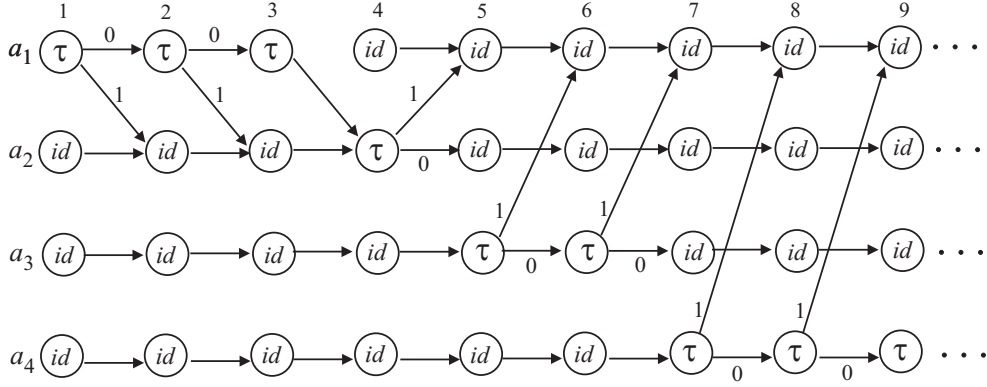}
\caption{The automaton  from Proposition~\ref{p2} for $n=4$, $d=3$, $r_1=3$, $r_2=1$, $r_3=2$.}
\label{fig}
\end{figure}

Let $a_{j,i}:=(a_j)_i$ ($1\leq j\leq n$, $i\in\no$) be the automaton transformation of the state $a_j\in Q$ in its $i$-th transition. Directly by the above formulae, we have the following wreath recursions for the automorphisms $a_{1,i}$ ($i\in\no$):
\begin{equation}\label{wst}
a_{1,i}=\left\{
\begin{array}{ll}
(a_{1,i+1} ,a_{2,i+1})\tau, &{\rm if}\;\;i\in N_1\setminus\{r_1\},\\
(a_{2,r_1+1}, a_{2,r_1+1})\tau, &{\rm if}\;\;i=r_1,\\
(a_{1,i+1}, a_{1,i+1}), &{\rm if}\;\;i\notin N_1.
\end{array}
\right.
\end{equation}
This implies: $a_{1,i}=id$ for $i>r_1$. Similarly, for every $2\leq j\leq n$ the wreath recursions for $a_{j,i}$ ($i\in\no$) satisfy: $a_{j,i}=(a_{j,i+1}, a_{1,i+1})\tau$ for every $i\in N_j$ and $a_{j,i}=(a_{j,i+1}, a_{j,i+1})$ for $i\notin N_j$. If $j\neq 1$ and $i\in N_j$, then $i\notin N_1$ and hence $i>r_1$. Consequently,  we have:
\begin{equation}\label{a}
a_{j,i}=\left\{
\begin{array}{ll}
(a_{j,i+1}, id)\tau, &{\rm if}\;\;j\neq 1\;\;{\rm and}\;\;i\in N_j,\\
(a_{j,i+1}, a_{j,i+1}), &{\rm if}\;\;j\neq 1\;\;{\rm and}\;\; i\notin N_j.
\end{array}
\right.
\end{equation}
By~(\ref{a}), we obtain for all $i, s\in\no$ and every $j\neq 1$:
$$
a_{j,i}^s=\left\{
\begin{array}{ll}
(a_{j,i+1}^{\lceil s/2\rceil},a_{j,i+1}^{\lfloor s/2\rfloor})\tau^s, &{\rm if}\;i\in N_j,\\
(a_{j,i+1}^s, a_{j,i+1}^s), &{\rm if}\; i\notin N_j.
\end{array}
\right.
$$
The  above formula implies  that the generator $a_{j,1}\in G(A)$ ($j\neq 1$) is of infinite order if and only if $|N_j|=\infty$, and if $|N_j|<\infty$ (which implies  $|N_j|=r_j$ by our assumption), then $o(a_{j,1})=2^{r_j}$.

Since $a_{2,i}=(a_{2,i+1}, a_{2,i+1})$ for  every $i\in N_1$, the wreath recursion of the element
$b_{1,i}:=a_{1,i}a_{2,i}^{-1}$ ($i\in N_1$)
satisfies:
\begin{equation}\label{b}
b_{1,i}=\left\{
\begin{array}{ll}
(b_{1, i+1}, id)\tau, &{\rm if}\;\;i\in N_1\setminus\{r_1\},\\
(id, id)\tau, &{\rm if}\;\; i=r_1.
\end{array}
\right.
\end{equation}
In particular, we obtain $o(b_{1,1})=2^{r_1}$.

By~(\ref{wst}) and~(\ref{a}), we also obtain the following condition:
if $a_{j,i}\neq (a_{j,i+1}, a_{j,i+1})$ for some $i\in\no$ and $1\leq j\leq n$, then $a_{j,i}=(a_{l,i+1}, a_{l', i+1})\pi$ for some $l,l'\in\{1,2,j\}$, $\pi\in\{id, \tau\}$, and then  for every $j'\neq j$ the element $a_{j',i}$ has the following wreath recursion: $a_{j',i}=(a_{j',i+1}, a_{j',i+1})$. This implies that for any $i\in\no$, $1\leq j,j'\leq n$, $x\in\{0,1\}$ there are $\pi\in\{id, \tau\}$ and $l,l'\in\{1,2, j, j'\}$ such that
\begin{eqnarray*}
a_{j,i}a_{j',i}|_x&=&a_{l,i+1}a_{l',i+1},\;\;\;\;\;a_{j',i}a_{j,i}|_x=a_{l',i+1}a_{l,i+1},\\
\sigma_{a_{j,i}a_{j',i}}&=&\sigma_{a_{j',i}a_{j,i}}=\pi.
\end{eqnarray*}
Consequently, the elements $a_{j,i}$ and $a_{j',i}$ commute for all $i\in\no$ and $1\leq j,j'\leq n$. Thus the group
$G(A)=\langle a_{1,1},\ldots, a_{n,1}\rangle=\langle b_{1,1}, a_{2,1}, \ldots, a_{n,1}\rangle$ is abelian.

Let $k_j$ ($1\leq j\leq n$) be arbitrary integers for which the product
$g=b_{1,1}^{k_1}\cdot a_{2,1}^{k_2}\cdot\ldots\cdot a_{n,1}^{k_n}$ represents the  trivial element in $G(A)$. By~(\ref{a})--(\ref{b}), we see that the root permutation $\sigma_g$  is equal to $\tau^{k_1}$. Hence $2\mid k_1$ and consequently $b_{1,1}^{k_1}=(b_{1,2}^{k_1/2}, b_{1,2}^{k_1/2})$. By~(\ref{a}), we have: $a_{j,1}^{k_j}=(a_{j,2}^{k_j}, a_{j,2}^{k_j})$ for  $2\leq j\leq n$. Thus $g|_0=g|_1=b_{1,2}^{k_1/2}\cdot a_{2,2}^{k_2}\cdot\ldots\cdot a_{n,2}^{k_n}$. Since $g|_w=id$ for every $w\in\{0,1\}^*$, a trivial induction on $i$ gives for every $i\in N_1$ and every  $w\in \{0,1\}^{i-1}$ the  divisibility $2^{i-1}\mid k_1$ and the equality $g|_w=b_{1,i}^{k_1/{2^{i-1}}}\cdot a_{2,i}^{k_2}\cdot\ldots\cdot a_{n,i}^{k_n}$. In particular $g|_w=b_{1,r_1}^{k_1/{2^{r_1-1}}}\cdot a_{2,r_1}^{k_2}\cdot\ldots\cdot a_{n,r_1}^{k_n}$ for every $w\in\{0,1\}^{r_1-1}$. Again by~(\ref{a})--(\ref{b}), we see that the root permutation of the last product is equal to $\tau^{k_1/{2^{r_1-1}}}$, which implies
$2^{r_1}\mid k_1$. Since $b_{1, r_1}|_0=b_{1,r_1}|_1=id$, we obtain: $g_w=a_{2,r_1+1}^{k_2}\cdot\ldots\cdot a_{n,r_1+1}^{k_n}$ for every $w\in \{0,1\}^{r_1}$ and hence:
\begin{equation}\label{g_w}
a_{2,r_1+1}^{k_2}\cdot\ldots\cdot a_{n,r_1+1}^{k_n}=id.
\end{equation}

Let us fix $i\in \no\setminus\{1,\ldots, r_1\}$. Let $l_j$ ($2\leq j\leq n$) be any integers for which the product
\begin{equation}\label{h}
h=a_{2,i}^{l_2}\cdot\ldots\cdot a_{n,i}^{l_n}
\end{equation}
represents the trivial element in the group $G(A_i)$. Since the sets $N_j$ ($2\leq j\leq n$) are pairwise disjoint, we have two possibilities: (i)  there is a unique $2\leq j_0\leq n$ such that $i\in N_{j_0}$, (ii) for every $2\leq j\leq n$ we have $i\notin N_j$. In the case (i), by the wreath recursion~(\ref{a}), we see that the root permutation $\sigma_h$ is equal to $\tau^{l_{j_0}}$. Since $h=id$, we obtain: $2\mid l_{j_0}$, and consequently, the sections $h|_0$, $h|_1$ coincide and they both arise from~(\ref{h}) by replacing the factor $a_{j_0, i}^{l_{j_0}}$ with $a_{j_0, i+1}^{l_{j_0}/2}$ and the factors $a_{j, i}^{l_{0}}$ for $j\neq j_0$ with $a_{j, i+1}^{l_{j}}$, i.e. we have:
$$
h|_0=h|_1=a_{2,i+1}^{l_2}\cdot\ldots\cdot a_{j_0-1, i+1}^{l_{j_0-1}}\cdot a_{j_0, i+1}^{l_{j_0}/2}\cdot a_{j_0+1, i+1}^{l_{j_0+1}}\cdot\ldots \cdot a_{n,i+1}^{l_n}.
$$
In the  case (ii) the sections $h|_0$, $h|_1$ both arise from $h$ by replacing each $a_{j, i}^{l_{j}}$ ($2\leq j\leq n$) with $a_{j, i+1}^{l_{j}}$, i.e. we have: $h|_0=h|_1=a_{2,i+1}^{l_2}\cdot \ldots \cdot  a_{n,i+1}^{l_n}$. Since $h|_w=id$ for every $w\in\{0,1\}^*$, we obtain by the trivial induction on $t\in\no$ that for every $t\in\no$ and every $2\leq j\leq n$ the number $l_j$ is divisible by $2^{\mu_{j,t}}$, where
$$
\mu_{j,t}=|N_j\cap \{i,i+1, \ldots, i+t-1\}|.
$$
Thus, if for some $2\leq j\leq n$ the set $N_j$ is infinite, then the sequence $(\mu_{j,t})_{t\in\no}$ is unbounded, which implies  $l_j=0$. If $|N_j|<\infty$ for some $2\leq j\leq n$, then we obtain the divisibility: $2^{\mu_j}\mid l_j$, where $\mu_j=|N_j\cap \{i'\in\no\colon i'\geq i\}|$. By our assumption, we have $N_j\cap \{1,2,\ldots, r_1\}=\emptyset$ and $|N_j|=r_j$ for every $2\leq j\leq n$. In particular, if $i=r_1+1$, then  we obtain in that case: $\mu_j=|N_j|=r_j$, and hence $o(a_{j,1})\mid l_j$. Applying this result to the equality~(\ref{g_w}), we obtain: $k_j=0$ for every $2\leq j\leq n$ such that $|N_j|=\infty$ and $o(a_{j,1})\mid k_j$ for every $2\leq j\leq n$ such that $|N_j|<\infty$. Consequently, we obtain
$$
G(A)\simeq \langle b_{1,1}\rangle\oplus \langle a_{2,1}\rangle\oplus\ldots\oplus\langle a_{n,1}\rangle\simeq G.
$$
\qed
\end{proof}

Now, it remains to consider the case of   abelian groups of rank $n$ ($n\geq 2$) without torsion elements. Note that every such a group must be the free abelian group of rank $n$, which follows from the property that every finitely generated abelian group is a direct sum of cyclic groups.

\begin{proposition}\label{p3}
If $n\geq 2$, then there is an $n$-state time-varying automaton $A$ over the binary alphabet which generates a group  isomorphic to  the free abelian group of rank $n$.
\end{proposition}
\begin{proof}
Let $n\in\no\setminus\{1\}$ and let $N_2\cup\ldots\cup N_{n}$
be an arbitrary partition of the set $\no\setminus\{1\}$ in which $|N_j|=\infty$ for every $2\leq j\leq n$.

Let $A$ be an $n$-state  time-varying automaton over the binary alphabet in which the set $Q$ of states consists of the symbols $a_1, \ldots, a_n$, and the sequences $\varphi=(\varphi_i)_{i\in\no}$, $\psi=(\psi_i)_{i\in\no}$ of transition and output functions are defined as follows:
\begin{eqnarray*}
\varphi_i(a_j,x)&=&\left\{
\begin{array}{ll}
a_1,&{\rm if}\;j\neq 1,\;i\in N_j,\;x=1,\\
a_2,&{\rm if}\;j=1,\;i=1,\;x=1,\\
a_j,&{\rm otherwise},
\end{array}
\right.\\
\psi_i(a_j,x)&=&\left\{
\begin{array}{ll}
\tau(x),&{\rm if}\;j\neq 1,\;i\in N_j,\\
x,&{\rm otherwise}
\end{array}
\right.
\end{eqnarray*}
for all $i\in\no$, $j\in\{1,2,\ldots, n\}$, $x\in\{0,1\}$.

As before, we denote by $a_{j,i}:=(a_j)_i$ ($1\leq j\leq n$, $i\in\no$) the automaton transformation of the state $a_j\in Q$ in its $i$-th transition. By the above formulae, we have the following wreath recursions for the automorphisms $a_{1,i}$ ($i\in\no$):
$$
a_{1,i}=\left\{
\begin{array}{ll}
(a_{1,i+1} ,a_{2,i+1}), &{\rm if}\;\;i=1,\\
(a_{1,i+1}, a_{1,i+1}), &{\rm if}\;\;i\neq 1.
\end{array}
\right.
$$
This implies that $a_{1,i}=id$ for $i\neq 1$, and consequently we obtain:
\begin{equation}\label{a2}
a_{1,i}=\left\{
\begin{array}{ll}
(id ,a_{2,i+1}), &{\rm if}\;\;i=1,\\
(id, id), &{\rm if}\;\;i\neq 1.
\end{array}
\right.
\end{equation}
The wreath recursions for $a_{j,i}$ ($i\in\no$, $j\neq 1$) satisfy~(\ref{a}), and similarly as in the previous proof, we obtain that the group $G(A)$ is abelian. Let $k_j$ ($1\leq j\leq n$) be arbitrary integers for which the product
$g=a_{1,1}^{k_1}\cdot\ldots\cdot a_{n,1}^{k_n}$ represents the  trivial element in $G(A)$. By~(\ref{a}) and by~(\ref{a2}), we obtain:
\begin{eqnarray}
g|_0&=&a_{2,2}^{k_2}\cdot\ldots\cdot a_{n,2}^{k_n}, \label{g0}\\
g|_1&=&a_{2,2}^{k_1+k_2}\cdot\ldots\cdot a_{n,2}^{k_n}.\label{g1}
\end{eqnarray}
Since $g|_0=id$, we can apply to the product~(\ref{g0}) the same reasoning as in the previous proof and obtain that for every $t\in\no\setminus\{1\}$ and every $2\leq j\leq n$ the number $k_j$ is divisible by $2^{\mu_{j,t}}$, where $\mu_{j,t}=|N_j\cap \{2,\ldots, t\}|$. Similarly,  since $g|_1=id$,  we use the product~(\ref{g1}) and obtain that the number $k_1+k_2$ is divisible by $2^{\mu_{2,t}}$ for every $t\in\no\setminus\{1\}$. Since each  $N_j$ is an infinite set, we obtain: $k_j=0$ for every $1\leq j\leq n$. Thus $G(A)$ is isomorphic to the free abelian group of rank $n$, which finishes the proof of the proposition. \qed
\end{proof}

Finally, we need to observe that if $G$ is a group generated by an $n$-state ($n\in\no$) time-varying (Mealy) automaton $A$ over a finite alphabet $X$, then for every $n'>n$ the group $G$ is also generated by an $n'$-state time-varying (Mealy) automaton over $X$. Indeed,  the required $n'$-state automaton is obtained by adding certain new symbols $q_1, \ldots, q_{n'-n}$ to the set of states of $A$ together with the following formulae for transition and output functions: $\varphi_i(q_j,x)=q_j$, $\psi_i(q_j, x)=x$ for all $i\in\no$, $j\in\{1,\ldots, n'-n\}$, $x\in X$. Then the automaton transformations $(q_j)_i\in Aut(X^*)$ ($i\in\no$, $1\leq j\leq n'-n$) are all trivial, and the  transformations corresponding to the ``old'' states remain unchanged. Hence, by Propositions~\ref{p1}--\ref{p3}, we obtain the inclusion $\mathcal{TVA}(n)\supseteq \mathcal{AB}_2(n)$ for every $n>1$, which finishes the proof of Theorem~\ref{t1}.

\section{The proof of Theorem~\ref{t2}}\label{sec6}

Let $G=G(A)$ be the   group generated by an arbitrary Mealy automaton $A$ over the binary alphabet and let us denote the following subsets of $G$:
\begin{itemize}
\item $G_1=\{g\in G\colon \sigma_g=id\}=\{g\in G\colon \forall {x\in\{0,1\}}\;g(x)=x\}$ -- the first level stabilizer,
\item $I_G=\{g\in G\colon g^2=id\}$ -- the set of involutions,
\item $G^2=\{g^2\colon g\in G\}$ -- the set of squares in $G$.
\end{itemize}

\begin{lemma}\label{l1}
If $G^2=G_1^2$, then $G$  is  an elementary abelian 2-group.
\end{lemma}
\begin{proof}
Suppose that $G^2=G_1^2$. We have to show that $I_G=G$. So, let $g\in G$ be arbitrary. It can be seen by the formula~(\ref{gw}) that the equality $g^2=id$ follows from the following condition: $g^2|_w\in G_1$ for every $w\in \{0,1\}^*$. But, since $G_1^2\subseteq G_1$, it is enough to show that $g^2|_w\in G_1^2$ for every $w\in \{0,1\}^*$. We use induction on the length of a word $w$. The claim is obvious in the case $|w|=0$, as then we have $g^2|_w=g^2|_\epsilon=g^2\in G^2=G_1^2$. Suppose that $g^2|_w\in G_1^2$ for some $w\in\{0,1\}^*$. Then $g^2|_w=h^2$ for some $h\in G_1$. Since $h=(h|_0, h|_1)$, we obtain: $g^2|_w=((h|_0)^2, (h|_1)^2)$,
and consequently $g^2|_{wx}=(g^2|_w)|_x=(h|_x)^2\in G^2=G_1^2$ for $x\in\{0,1\}$. An inductive argument finishes the proof of Lemma~\ref{l1}.\qed
\end{proof}

\begin{lemma}\label{l2}
If $I_G\subseteq G_1$, then  $I_G=\{id\}$.
\end{lemma}
\begin{proof}
Suppose that $I_G\subseteq G_1$ and let $g\in I_G$ be arbitrary. We must show that $g=id$. As before, we see by the formula~(\ref{gw}) that we need to show the condition: $g|_w\in G_1$ for every $w\in \{0,1\}^*$. Since $I_G\subseteq G_1$, it is enough to show that $g|_w\in I_G$ for every $w\in \{0,1\}^*$. We use  induction on the length of a word $w$. The claim is obvious in the case $|w|=0$. Suppose that $g|_w\in I_G$ for some $w\in\{0,1\}^*$. Since $I_G\subseteq G_1$, we have: $g|_w=(g|_{w0}, g|_{w1})$. Hence $id=(g|_w)^2=((g|_{w0})^2, (g|_{w1})^2)$. Thus $(g|_{wx})^2=id$ for $x\in\{0,1\}$, i.e. $g|_{wx}\in I_G$ for $x\in \{0,1\}$. An inductive argument finishes the proof.  \qed
\end{proof}

\begin{proposition}\label{pp}
If the group $G$ is abelian, then $G$ is torsion free or it is an elementary abelian 2-group.
\end{proposition}
\begin{proof}
Suppose that the group $G$ is non-trivial, abelian and not torsion free. Then there is $g\in G$ with $1<o(g)<\infty$. By Proposition~\ref{inclu}, there is $r\in\no$ such that $o(g)=2^r$. Then $g^{2^{r-1}}$ is a non-trivial element from $I_G$, and consequently,  by Lemma~\ref{l2}, the set $I_G\setminus G_1$ is not empty. Let $h\in I_G\setminus G_1$. For every $g\in G$ we have: if $g\notin G_1$, then $gh\in G_1$ and $g^2=g^2h^2=(gh)^2\in G_1^2$. Consequently, we obtain $G^2=G_1^2$ and,  by Lemma~\ref{l1}, the group $G$ is an elementary abelian 2-group, which finishes the proof of Proposition~\ref{pp}. \qed
\end{proof}

\begin{remark}
The statement of Proposition~\ref{pp} can also be derived by using the concept of a $1/2$-endomorphism of a group, which defines its state-closed representation on the tree $\{0,1\}^*$ -- see~\cite{16} and Proposition~3.4 therein.
\end{remark}

Now, the inclusion $\mathcal{MA}(n)\subseteq \mathcal{FA}(n-1)\cup \mathcal{EA}_2(n)$ ($n\geq 2$) follows directly by Proposition~\ref{pp} and by the following proposition (see Proposition~3.1 in~\cite{17}):
\begin{proposition}[\cite{17}]
For every $n\in\no$ there is no $n$-state  Mealy automaton over the binary alphabet which generates the free abelian group of rank $n$.
\end{proposition}

Since the equality $\mathcal{MA}(1)=\mathcal{EA}_2(1)$ is trivial, we see by the observation at the end of the previous section that to show for every $n\geq 2$ the converse  inclusion (i.e. $\mathcal{MA}(n)\supseteq \mathcal{FA}(n-1)\cup \mathcal{EA}_2(n)$), we only need to  show that the free abelian group of rank $n-1$ as well as the elementary abelian 2-group of rank $n$ is generated by an $n$-state Mealy automaton over the binary alphabet.
The case of a free abelian group follows directly by the following well known construction of  so-called ``sausage'' automata (see~\cite{2} for example):

\begin{proposition}[\cite{2}]
Let $A=(\{0,1\}, Q, \varphi, \psi)$ be an $n$-state ($n\geq 2$) Mealy automaton in which $Q=\{a_1,\ldots, a_n\}$ and the transition and output functions are defined as follows:
$$
\varphi(a_j, x)=\left\{
\begin{array}{ll}
a_{j-1}, &if\;\;2<j\leq n,\\
a_n, &if\;\;j=2,\; x=0,\\
a_1, &if\;\;j=2,\; x=1,\\
a_1, &if\;\;j=1,
\end{array}
\right.\;\;\;
\psi(a_j, x)=\left\{
\begin{array}{ll}
\tau(x), &if\;\;j=2,\\
x, &otherwise
\end{array}
\right.
$$
for all $j\in\{1,2,\ldots, n\}$ and $x\in\{0,1\}$. Then the group $G(A)$ is isomorphic to the free abelian group of rank $n-1$.
\end{proposition}

The case of an elementary abelian 2-group follows by the following
\begin{proposition}
Let  $A=(\{0,1\}, Q, \varphi, \psi)$ be an $n$-state ($n\in\no$) Mealy automaton in which $Q=\{a_1,\ldots, a_n\}$ and  the transition and output functions are defined as follows:
$$
\varphi(a_j, x)=\left\{
\begin{array}{ll}
a_{j-1}, &if\;\;2\leq j\leq n,\\
a_n, &if\;\;j=1,
\end{array}
\right.\;\;\;
\psi(a_j, x)=\left\{
\begin{array}{ll}
\tau(x), &if\;\;j=1,\\
x, &otherwise
\end{array}
\right.
$$
for all $j\in\{1,2\ldots, n\}$ and $x\in\{0,1\}$. Then $G(A)$ is isomorphic to the elementary abelian 2-group of rank $n$.
\end{proposition}
\begin{proof}
To simplify notation we will use the notation $a_j$ ($1\leq j\leq n$) for the automaton transformation $(a_j)_1\in G(A)$. By the formulae for transition and output functions, we have the following wreath recursions for the generators $a_j$:
\begin{equation}\label{11}
a_j=\left\{
\begin{array}{ll}
(a_{j-1}, a_{j-1}), &{\rm if}\;\;2\leq j\leq n,\\
(a_n, a_n)\tau, &{\rm if}\;\;j=1.
\end{array}
\right.
\end{equation}
In particular, for any $1\leq j,j'\leq n$ and any $x\in\{0,1\}$ there are $1\leq i,i'\leq n$ and $\pi\in\{id, \tau\}$ such that
$(a_ja_{j'})|_x=a_{i}a_{i'}$, $(a_{j'}a_{j})|_x=a_{i'}a_{i}$ and $\sigma_{a_ja_{j'}}=\sigma_{a_{j'}a_j}=\pi$. Thus for any $1\leq j,j'\leq n$ the elements $a_j$, $a_{j'}$ commute, and hence the group $G(A)$ is abelian. For every $1\leq j\leq n$ we have: $a_j(0^n)=0^{j-1}10^{n-j}$, and hence $a_j\neq id$. On the other hand, by~(\ref{11}), we see that for every $1\leq j\leq n$ there is $1\leq i\leq n$ such that $a_j^2=(a_i^2, a_i^2)$. Thus $a_j^2=id$ for every $1\leq j\leq n$. Finally, suppose that for some integers $k_j$ ($1\leq j\leq n$) the product $g=a_1^{k_1}\cdot\ldots\cdot a_n^{k_n}$ represents the trivial element in the group $G(A)$. Then the root permutation $\sigma_g$ is equal to $\tau^{k_1}$, and hence $2\mid k_1$. We also see by wreath recursions~(\ref{11}) that for every $t\in\{0,\ldots, n-1\}$ and every word $w\in \{0,1\}^t$ the section $g|_w$ is equal to
$$
a_1^{k_{t+1}}\cdot a_2^{k_{t+2}}\cdot\ldots\cdot a_{n-t}^{k_n}\cdot a_{n-t+1}^{k_1}\cdot\ldots\cdot a_n^{k_t}.
$$
Thus  $id=\sigma_{g|_w}=\tau^{k_{t+1}}$, and consequently $2\mid k_{t+1}$, which implies the isomorphism $G(A)\simeq \langle a_1\rangle\oplus\ldots\oplus\langle a_n\rangle\simeq C_2^n$.\qed
\end{proof}

\end{document}